\documentclass[a4paper,10pt]{amsart}

\usepackage[T2A]{fontenc}
\usepackage{inputenc}
\usepackage{amssymb}
\usepackage[matrix,arrow,curve]{xy}
\usepackage{latexsym}
\usepackage{amscd,amsthm,amsfonts,amssymb,amsmath,euscript}

\DeclareMathAlphabet{\mathbbold}{U}{bbold}{m}{n}
\def \k {\mathbbold{k}}
\newtheorem{theorem}[equation]{Theorem}
\newtheorem{proposition}[equation]{Proposition}

\theoremstyle{definition}
\newtheorem{definition}[equation]{Definition}

\theoremstyle{remark}

\theoremstyle{definition}

\newcommand{\Gal}{{\mathrm{Gal}}}

\newcommand{\GG}{{G}}

\newcommand{\PP}{{\mathbb P}}

\newcommand{\Pic}{{\mathrm{Pic}}}

\newcommand{\XX}{{\overline X}}

\title{Quasitriviality of the forms of Segre varieties}
\author{N.F. Zak}
\email{nzak@mccme.ru}
\thanks{The author was partially supported by the RFFI grant 04-01-00613}

\begin{document}
\date{}
\maketitle

\begin{abstract}
We prove the rationality of a $\k$-form $X$ of the product $S$ of projective spaces provided the existence of a
$\k$-point on $X$. The method of the proof is to find a Galois-invariant birational projection of $S$ to the
projective space. This method also allows to prove the quasitriviality of the forms of the hyperplane sections of
some Segre varieties.
\end{abstract}

\bigskip

Consider a variety $X$ defined over a field $\k$ of characteristic zero. The corresponding variety
$X\times_\k\overline \k$ over the algebraic closure $\overline \k$ of the field $\k$ we denote by $\XX$ saying that
$X$ is a form of any variety isomorphic to $\overline X$ over $\overline \k$.
Let $\GG=\Gal(\overline \k/\k)$ be the Galois group. 
We will use the following fact below (see [2]):

\begin{proposition}\label{grot}
Let $X$ be a variety with a $\k$-point and $\Pic^{G}(\XX)$ be the Galois-invariant part of the Picard group
$\Pic(\XX)$. Then there is a group isomorphism $\Pic(X)\simeq\Pic^G(\XX)$ and every Galois-invariant linear system
is generated by the divisors defined over $\k$.
\end{proposition}

Also, it is easy to check that the Galois-invariant part of a very ample linear system on the variety with a
$\k$-point defines an embedding over $\k$. Recall the following definitions (see [1]).

\begin{definition}
The variety $X$ is called $\k$-rational or rational over $\k$ if the function field $\k(X)$ is a pure transcendent
extension. $X$ is called rational if $\XX$ is rational over $\overline \k$.
\end{definition}

\begin{definition}
The class of rational varieties defined over $\k$ such that for each variety from this class the existence of a
$\k$-point implies $\k$-rationality, is called {\it quasitrivial over $\k$} and the varieties from this class are
also called {\it quasitrivial over $\k$}.
\end{definition}

It is well known that each form $X$ of the projective space $\PP^n$ is quasitrivial, moreover, Severi--Brauer
theorem states that the existence of a $\k$-point on the form $X$ of the projective space $\PP^n$ implies the
isomorphism $X\simeq_\k\PP^n$. Analogous result does not hold for the products of projective spaces. For example,
the two-dimensional real quadric given by equation $x_0^2+x_1^2+x_2^2=x_3^2$ has real points but is not isomorphic
to the product of real projective lines. Nevertheless, such varieties are still quasitrivial. We prove the following

\begin{theorem}\label{th}
Each form $X$ of the variety $S=\PP^{a_1}\times\ldots\times\PP^{a_n}$ is quasitrivial.
\end{theorem}

The following Proposition gives a method (which is far from being optimal, of course) of proving the rationality of
$S$.

\begin{proposition}\label{proj}
Consider standard Segre embedding of $S$. Pick a hyperplane $H_i$ in the $i$-st multiplier and consider the
subvarieties
$$
\mathcal{P}_{ij}=\PP^{a_1}\times\ldots\times\PP^{a_{i-1}}\times
H_i\times\PP^{a_{i+1}}\times\ldots\times\PP^{a_{j-1}}\times H_j\times\PP^{a_{j+1}}\times\ldots\times\PP^{a_n},
\mbox{ } i<j.
$$
Let $L$ be the projective span of $\mathcal P=\bigcup{\mathcal{P}_{ij}}$. The projection $\pi_L :
S\dashrightarrow\PP^{\sum{a_i}}$ is birational. In particular, if $S$ and $L$ are defined over $\k$ then the
projection $\pi_L$ is birational over $\k$.
\end{proposition}

\begin{proof}
The finiteness of the projection $\pi_L : S\dashrightarrow\PP^{\sum{a_i}}$ can be checked by counting the dimension
of $L$. Note that through every pair of points $p,q\in S$ their passes a subvariety
$$
S_1=T_1\times\ldots\times T_n\subset S,\mbox{    where    } T_i=\PP^1\subset\PP^{a_i},
$$
so, due to the finiteness of the general fiber of $\pi_L$, it is sufficient to prove birationality of projection
$\pi_L: S_1\to\nolinebreak\PP^n$ for a general point $p$ and any point $q$. For such a pair of points we have
$\mathcal P\cap S_1=\bigcup{P_{ij}}$ and, moreover, $L\cap S_1=\bigcup{P_{ij}}$ where

$$
P_{ij}=T_1\times\ldots\times T_{i-1}\times p_i\times T_{i+1}\times\ldots\times T_{j-1}\times p_j\times
T_{j+1}\times\ldots\times T_n,\mbox{ } i<j
$$
and $p_i=H_i\cap T_i$. Thus it suffices to prove the birationality of projection $\pi_{L_1}:
S_1\dashrightarrow\PP^n$ where $L_1$ is the linear span of $P=\bigcup{P_{ij}}$. Consider a general point
$q=q_1\times\ldots\times q_n\in S_1$. Let
$$
Q_i=T_1\times\ldots\times T_{i-1}\times q_i\times T_{i+1}\times\ldots\times T_n, $$$$P_j=T_1\times\ldots\times
T_{j-1}\times p_j\times T_{j+1}\times\ldots\times T_n.
$$
Consider $n$ independent hyperplane sections ${\mathcal H}_i\subset S_1$ containing $L_1$ and $q$:
$$
{\mathcal H}_i=Q_i\cup P^i, \mbox{ where } P^i=\bigcup_{j\ne i}P_j.
$$
Let $\mathcal H=\bigcap \mathcal{H}_i$. The birationality of projection $\pi_{L_1}$ is equivalent to the equality
$$
\mathcal H\cap S_1=\{q\}\cup P.
$$
Consider a point $r=r_1\times\ldots\times r_n\in\mathcal H\cap S_1$. Since $r\in Q_i\cup P^i$, either $r_i=q_i$ or
$r_k=p_k$ for some $k$. If the first equality holds for each $i$ then $r=q$. If $r_k=p_k\ne q_k$ then, since $r\in
Q_k\cup P^k$, for some $l\ne k$ holds $r_l=p_l$, so $r\in P_{kl}\subset P$ which proves the birationality of
$\pi_{L_1}$.
\end{proof}

\begin{proof}[The proof of Theorem \ref{th}] Since the Galois group can not permute the contractions
to the varieties of different dimensions, $\k$-form $X$ of the variety $S$ is isomorphic to the product of
$\k$-forms of the varieties $S_i=\PP^{a_i}\times\ldots\times\PP^{a_i}$ and it is sufficient to prove the
quasitriviality of these multipliers.

Consider variety $S = \PP^n\times\ldots\times\PP^n$. The linear system defining Segre embedding of $S$ is
Galois-invariant because it is proportional to the canonical class $K_S$. Hence, due to Proposition \ref{grot} we
can assume  that $\XX$ is embedded by Segre. We will construct the Galois-invariant union $\bigcup H_i$ so that
$\bigcup \mathcal{P}_{ij}$ is also Galois-invariant and by Severi--Brauer theorem Proposition \ref{proj} implies
quasitriviality of $X$.

Denote by ${\mathcal L_i}$ the linear system defining the natural projection of $S$ to the $i$-st multiplier. Let
$\k_i\supset \k$ be the minimal filed over which the linear system ${\mathcal L_i}$ is defined. Choose (using
Proposition \ref{grot}) a $\k$-divisor $D_1=H_1\times\PP^n\times\ldots\times\PP^n$ in the corresponding
Galois-invariant linear system.

Note that $\GG$-orbit $O_1$ of the divisor $D_1$ consists of such divisors $D_i$ that $D_i\nsim D_j$ in the group
$\Pic(\XX)$ for $i\ne j$. Otherwise, there exists such an element $g\in G$ that $g(D_i)\ne D_i$ but $g(D_i)\sim D_i$
for some $i$  and a conjugate to $g$ element $h\in G$ such that $h(D_1)\ne D_1$ but $h(D_1)\sim D_1$. Since in this
case $h$ acts trivially on $\k_1$ and $D_1$ is defined over $\k_1$, we obtain a contradiction.

If in the orbit $O_1$ there is no divisors from the linear system ${\mathcal L_j}$, consider the orbit $O_j$ of the
corresponding divisor $D_j$, and so on for all $j$. The union of such orbits will give us the invariant center $L$
of the projection which proves Theorem \ref{th}.
\end{proof}
The proof above may be generalized to the hyperplane sections of Segre varieties.
\begin{proposition}
Each form $Y$ of the hyperplane section $W=H\cap S\subset\PP^{ab+a+b-1}$ of Segre variety
$S=\PP^a\times\PP^b\subset\PP^{ab+a+b}$ is quasitrivial.
\end{proposition}
\begin{proof}
Note that the pair of linear systems defining the projections of $W$ to the multipliers $\PP^a$ and $\PP^b$ is
Galois-invariant (the Galois group can only transpose the projections to different multipliers in the case $a=b$).
Thus, the natural inclusion $W\hookrightarrow S$ is Galois-invariant and by Proposition \ref{grot} we can assume
that $\overline Y=W=H\cap S\subset\PP^{ab+a+b-1}$ where $S$ and $H$ are defined over $\k$.

Consider the invariant center $L$ of the projection $\pi_L: S\dashrightarrow\PP^{a+b}$ and the birational projection
$$
\pi_{L_H}: S\dashrightarrow Q\subset\PP^{a+b+1}
$$
with center $L_H=L\cap H$ to the hypersurface $Q\subset\PP^{a+b+1}$. Over $\overline{\k}$, through every point $p\in
L_H\cap S$ there passes $(a+b-2)$-dimensional family of the subvarieties of type $\PP^1\times\PP^1\subset S$ and
general subvariety of this type is projected to a plane on $Q$ via $\pi_{L_H}$. Therefore, $Q$ contains a
$(2(a+b)-5)$-dimensional family of planes and its hyperplane section $Q_H=\pi_{L_H}(W)$ has dimension $a+b-1$ and
contains $(2(a+b-1)-3)$-dimensional family of lines. Thus, according to [3], hypersurface $Q_H$ is either a quadric
or a rational (due to rationality of $W$) scroll in $P^{a+b}$, that is a family of hyperplanes parameterized by a
rational curve.

Quasitriviality of quadric is well-known. Let $Q_H$ be a rational scroll. Due to the generality of $L$, the
existence of a smooth $\k$-point on $Q_H$ follows from the existence of a $\k$-point on $W$. In this case $Q_H$ is
birational to the projective bundle over $\PP^1$ with a fiber $\PP^{a+b-2}$. General codimension $(a+b-2)$ subspace
gives a rational section to this bundle, so the hypersurface $Q_H$ is quasitrivial and, due to the birationality of
the projection $\pi_{L_H}$, the variety $W$ is also quasitrivial.
\end{proof}

\medskip

The author is grateful to S.\,Galkin, S.\,Gorchinskiy, V.\,Iskovskikh, Yu.\,Prokhorov, C.\,Shramov and F.\,Zak for
useful discussions.

\end{document}